\documentclass[secnum,leqno]{rims-bessatsu}%%% eqno resets every section 
\usepackage{amssymb, amsmath}%%% option packages
\usepackage{graphics}%%% option packages
\usepackage[dvips]{graphicx}%%% option packages
\usepackage{eepic}%%% option packages
\usepackage{epic}%%% option packages
\usepackage{bm}
\newtheorem{thm}{Theorem}[section]
\newtheorem{crl}[thm]{Corollary}

\newtheorem{lmm}[thm]{Lemma}

\newtheorem{prp}[thm]{Proposition}

\theoremstyle{definition}
\newtheorem{dfn}[thm]{Definition}
\newtheorem{exa}[thm]{Example}%%%
\theoremstyle{remark}
\newtheorem*{rem}{Remark}%%% * means no numbering
\title{Tree representations of $ \alpha$-determinantal point processes}
\author{Shota \textsc{OSADA}\footnote{Graduate School of Mathematics, Kyushu University, Fukuoka 
819-0395, Japan.\newline e-mail: \texttt{s-osaad@math.kyushu-u.ac.jp}}
}
\AuthorHead{Shota OSADA}         %optional but recommended
\support{Supported by Grant-in-Aid for JSPS Research Fellow 18J20465}  %optional
         %optional
\VolumeNo{x}           %editors must define this.
\YearNo{201x}           %editors must define this.
\PagesNo{000--000}      %editors must define this.
\begin{document}
%
% The text goes here.  
% Be sure to use the appropriate "theorem-like" environment as 
% is the following examples.  Never use plain TeX commands for these, as
% they will cause interference with the styles of other papers. 

\maketitle

%\tableofcontents      %optional
\begin{abstract}      %optional
We introduce tree representations for $ \alpha$-determinantal point processes. 
The $ \alpha$-determinantal point processes is introduced in \cite{s-t.jfa} 
as a one parameter extension of the determinantal point process. 
In \cite{o-o.tail}, 
the tree representation was introduced for determinantal point processes. 
In this paper, we prove that the tree representation can be applied to $ \alpha$-determinantal point processes. 
\end{abstract}

\newcommand\DN{\newcommand}\newcommand\DR{\renewcommand}

\DN\lref[1]{Lemma~\ref{#1}}
\DN\tref[1]{Theorem~\ref{#1}}
\DN\pref[1]{Proposition~\ref{#1}}
\DN\sref[1]{Section~\ref{#1}}
\DN\ssref[1]{Subsection~\ref{#1}}
\DN\dref[1]{Definition~\ref{#1}}
\DN\rref[1]{Remark~\ref{#1}} 
\DN\corref[1]{Corollary~\ref{#1}}
\DN\eref[1]{Example~\ref{#1}}
\numberwithin{equation}{section}
\newcounter{Const} \setcounter{Const}{0}

%\DN\cref[1]{c _{#1}}  	%(before completion)%
%\DN\Ct{\mathsf{c}}	%(before completion)%
%\def\cref#1{c _{#1}}

%\DN\Ct[1]{\refstepcounter{Const}c_{\theConst}\label{#1}}
\DN\Ct{\refstepcounter{Const}c_{\theConst}}
\numberwithin{Const}{section}
%(before completion)%

%\DN\Ct[1]{\refstepcounter{Const}c_{\theConst}\label{#1}}
%\numberwithin{Const}{section}
%(after completion)%

%%% \DN\cref[1]{c_{\ref{#1}}}	
%(after completion)%
%%%   ueno "before" wo "after" he torikae, sarani %%%%%%
%%% \Ct_{; wo saigoni sita ni kaeru!!!	%%%%%%
%%%  	\Ct \label{;  		%%%%%%

%\DN\Ct{\refstepcounter{Const}c_{\theConst}}%\label{#1}
%\DN\Ct[1]{\refstepcounter{Const}c_{\theConst}\label{#1}}

%\DN\cref[1]{c_{\ref{#1}}} %%after
\DN\cref[1]{c_{\ref{#1}}}	%%before

% kigou
\DN\R{\mathbb{R}}
\DN\N{\mathbb{N}}
\DN\Q{\mathbb{Q}}
\DN\C{\mathbb{C}}
\DN\Z{\mathbb{Z}}

\DN\map[3]{#1\!:\!#2\!\to\!#3}
\DN\ot{ \otimes } 
\DN\ts{ \times }

\DN\limi[1]{\lim_{#1\to\infty}} 	
\DN\limz[1]{\lim_{#1\to0}}
\DN\limsupi[1]{\limsup_{#1\to\infty}} 	
\DN\limsupz[1]{\limsup_{#1\to0}}
\DN\liminfi[1]{\liminf_{#1\to\infty}} 	
\DN\liminfz[1]{\liminf_{#1\to0}}

\DN\sumii[1]{\sum_{#1=1}^{\infty}}
\DN\sumi[1]{\sum_{#1=0}^{\infty}}

\DN\PD[2]{\frac{\partial#1}{\partial#2}}
\DN\half{\frac{1}{2}}
\DN\Rd{\R ^d}

\DN\bs{\bigskip}
\DN\ms{\smallskip}

%===================================

\DN\Jli{J _{\ell ,\ii }}
\DN\Jonei{J _{1,\ii }}
\DN\Jtwoi{J _{2,\ii }}
\DN\rfrak{r}

%\DN\iii{\mathbf{i}}
%\DN\iii{\iit}
%\DN\tinyiii{\tinyiit}
\DN\iii{\bm{i}}
%\DN\jjj{\mathbf{j}}
\DN\jjj{\bm{j}}
\DN\iiim{(\iii _1,\ldots,\iii _m)}
\DN\iiM{\ii }

%\DN\Fl{\mathcal{F}_{\ell } }
\DN\Fl{\mathfrak{F}_{\ell}}
\DN\ulab{\mathfrak{u}}
\DN\ulabl{\mathfrak{u}_{\ell }}
\DN\ulabi{\mathfrak{u}_{\ell ,\mathsf{i}}}

\DN\FFF{\mathbb{F}} \DN\FFFl{\FFF (\ell )}\DN\FFFz{\FFF (0)}
\DN\Il{I (\ell )}

\DN\mmFliiis{\mmm _{\mm _{\ell , \iii _{\sigma }} }}

\DN\wIIIl{\wIII (\ell )} 
\DN\wIII{\widetilde{\III }}
\DN\mmFl{\mmm _{\FFFl }}

\DN\I{I}
\DN\lambdal{\lambda _{\IIIl }}
\DN\III{\mathbb{I}}
\DN\IIIxl{\III (\ell )_{\mathbf{x}}} 
\DN\IIIl{\III (\ell )} 
\DN\IIl{\mathsf{I}(\ell )}

\DN\AAAA{\mathcal{A}}
\DN\Ai{\mathcal{A}_i } \DN\Aj{\mathcal{A}_j }
\DN\Aii{\mathcal{A}_{i'} } \DN\Ajj{\mathcal{A}_{j'} }
\DN\Aiii{\mathcal{A}_{i^-} } \DN\Ajjj{\mathcal{A}_{j^-} }
 \DN\Aiin{\mathcal{A}_{\ii _n}} 
 \DN\Ali{\AAAA _{\ell , \ii } } 
 \DN\Aliii{\AAAA _{\ell , \ii ^-} } 
 \DN\Azi{\AAAA _{0 , \ii } } 
\DN\iim{(\ii _1,\ldots,\ii _m)}

\DN\B{\mathcal{B} }
\DN\Bi{\B _{\ii }} 
\DN\Bl{\B _{\ell }}
\DN\Bli{\B _{\ell , \ii }}
\DN\Blin{\B _{\ell , \ii _n}}
\DN\Blj{\B _{\ell , \jj }}
\DN\Bj{\B _{\jj }}
\DN\Bljn{\B _{\ell , \jj _n}}

%\DN\Gl{\mathcal{G}_{\ell }} \DN\Gll{\mathcal{G}_{\ell +1}}
\DN\Gl{\mathfrak{G}_{\ell }} \DN\Gll{\mathfrak{G}_{\ell +1}}

\DN\fliirs{f_{\ell ,\ii _{r,s}}}
\DN\fliir{f_{\ell ,\ii _r}}
\DN\flii{f_{\ell ,\ii }}\DN\fljj{f_{\ell , \jj }}
\DN\fliip{f_{\ell ,\iip }}
\DN\fliiq{f_{\ell ,\iiq }}
\DN\fliipsigma{f_{\ell ,\ii _{\sigma (p )} }}
\DN\fljjq{f_{\ell , \jjq }}
\DN\fljjp{f_{\ell , \jjp }}

\DN\fljjr{f_{\ell , \jj _r}}
\DN\fljjrs{f_{\ell , \jj _{r,s}}}
\DN\fzii{f_{0 ,\ii }}\DN\fzjj{f_{0 , \jj }}

\DN\fii{f_{\ii }}\DN\fjj{f_{\jj }}
\DN\fiil{f_{\ii '}}
\DN\fliil{f_{\ell , \ii '}}
\DN\ii{i}\DN\jj{j}

\DN\iip{\ii _p}
\DN\jjp{\jj _p}
\DN\jjq{\jj _q}
\DN\iiq{\ii _q}

\DN\rhoF{\rho _{\mathbb{F}}}
\DN\rhoFl{\rho _{\FFFl }}

\DN\rhoZ{\rho _{\Gz }}%\DN\rhoZ{\rho _{\Delta (0)}}
\DN\rhol{\rho _{\Gl }}%\DN\rhol{\rho _{\Delta (\ell )}}
\DN\nuF{\nu_{\mathbb{F}}}
\DN\nuFl{\nu_{\FFFl }}
\DN\KF{\K _{\mathbb{F}}}
\DN\KFl{\K _{\FFFl }}
\DN\KR{\K _{R}}
\DN\f{\mathbf{f}}
\DN\Det{\mathrm{Det}}

\DN\aaa{\mathsf{a}}
\DN\ijn{{_{i,j=1}^{n}}}
\DN\mub{\mu _{\mathrm{Be},\alpha }}
\DN\rgn{\rho _{\mathrm{Gin}}}
\DN\kg{\mathsf{K}_{\mathrm{Gin}}}

\DN\mug{\mu _{\mathrm{Gin}}}
\DN\mugx{\mu _{\mathrm{Gin},x}}
\DN\muSinb{\muone _{\mathrm{Sin}, \beta }}
\DN\muAi{\mu _{\mathrm{Ai}}}
\DN\lab{\mathfrak{l}}

%%%%%%%%%%%%%%%%

\DN\rairybeta{\rho _{\mathrm{Ai},\beta }}
\DN\rairytwo{\rho _{\mathrm{Ai}}} %\DN\rairytwo{\rho _{\mathrm{Ai},2}}
\DN\rNone{\rho ^{\n ,1}}
\DN\rNtwo{\rho ^{\n ,2}}
\DN\rNnk{\rho ^{\n ,n+k}}

\DN\rhobar{\bar{\rho }}
\DN\rb{\rho _{\beta }^n}
\DN\rbone{\rho _{\beta }^1}
\DN\rbx{\rho _{1 }^n}
\DN\rby{\rho _{2 }^n}
\DN\rbz{\rho _{4 }^n}

\DN\rbNone{\rho _{\beta }^{\n ,1}}
\DN\rbNtwo{\rho _{\beta }^{\n ,2}}
\DN\rbNxone{\rho _{\beta ,x}^{\n ,1}}
\DN\rbNxtwo{\rho _{\beta ,x}^{\n ,2}}
\DN\rbNxn{\rho _{\beta ,x}^{\n ,n}}
\DN\rN{\rho ^{\n ,n}}
\DN\rbN{\rho _{\beta }^{\n ,n}}
\DN\rbNN{\rho _{\beta }^{\n ,N}}
\DN\rbNx{\rho _{1}^{\n ,n}}
\DN\rbNy{\rho _{2}^{\n ,n}}
\DN\rbNz{\rho _{4}^{\n ,n}}

%%%%%%%%%%%%%%%%%%%

\DN\vxNi{v _{\xNi }} 	
\DN\vyNi{v _{\yNi }}
\DN\GN{G_{\nN }} 	
\DN\GNN{G_{\nN }\ts \N } 	
\DN\CGN{\C ^{\GN\ts \N }}
\DN\CN{\C ^{\N }} 	\DN\RN{\R ^{\N }}
\DN\xyGN{[x]_{\nN }, \, [y]_{\nN } \in \GN }
\DN\xGN{[x]_{\nN }  \in \GN }
\DN\yGN{[y]_{\nN }  \in \GN }

\DN\vv{\mathbf{v}}
\DN\axyN{\mathsf{K}_{x,y}^{\nN }}%\DN\axyN{a_{x,y}^{\nN }} 
\DN\MMN{\mathsf{M}^{\nN }}
\DN\MN{\mathsf{M}^{\nN }}
\DN\LLN{\LL ^{\nN }}
\DN\LLNx{\LLN _{x}}
\DN\LLNy{\LLN _{y}}
\DN\LLxNi{\LL _{\xNi }}
\DN\LLyNi{\LL _{\yNi }}
\DN\LLxi{\LL _{\xi }}
\DN\LLeta{\LL _{\eta }}

\DN\T{\mathbf{T}}
\DN\Tl{\T _{\ell }}
\DN\Tll{\T _{\ell '}}
\DN\TTT{\mathbb{T}}
\DN\Tail{\mathrm{Tail} }
\DN\nN{\mathsf{N}}
\DN\IN{I_{\nN }}
\DN\sS{\mathrm{S}}
\DN\sSS{\mathfrak{S}}
\DN\sss{\mathsf{s}}
%\DN\sss{\xi}
\DN\K{\mathsf{K}}
\DN\Kl{{\mathsf{K}}_{\ell }}
\DN\KKK{\mathsf{K}} 
\DN\LL{\mathsf{L}}
%\DN\mm{\mathsf{m}}
\DN\mm{ f }
%\DN\mmm{\mathsf{m}}
\DN\mmm{\lambda}
\DN\xN{[x]_{\nN }}
\DN\yN{[y]_{\nN }}

\DN\TailS{\Tail (\Ssf )}
\DN\Path{\mathsf{Path}}
\DN\Tree{\mathsf{Tree}}
\DN\Da{\underline{\Delta }}\DN\Db{\underline{\Delta }'}
\DN\DDa{\{\Dl \}_{\ell \in \{ 0 \}\cup \N} }

\DN\Ssf{\mathsf{S}}
%\DR\SS{\mathsf{S}}
\DN\rrr{\mathsf{r}}
\DN\Tf{\Tree _{\mathrm{fin}}}

\DN\rr{\mathsf{r}}
\DN\eN{\ell \in \N }
\DN\nnn{\mathsf{n}}
\DN\nNN{\nN _{\nnn }}
\DN\muN{\mu ^{\nnn }}

\DN\li{_{\ell }(i)}
\DN\kl{_{\ell }}
\DN\kll{_{\ell +1}}
\DN\kkl{_{\ell }}
\DN\Kli{\K _{\ell ,\infty}} 
\DN\Ki{\K _{\infty }}
\DN\Klbar{\K  \kl }
\DN\KDkl {\K \kl }
\DN\ml{\mm \kl } 
\DN\mlv{\mm _{\ell , v} } 
\DN\mli{\mm _{\ell }(i)} \DN\mlj{\mm _{\ell }(j)}

\DN\ff{\mathfrak{f}}

\DN\Klt{\widetilde{\K }_{\ell }}

\DN\Dz{\Delta (0)}\DN\Done{\Delta (1)}
\DN\Dl{\Delta (\ell ) } \DN\Dll{\Delta (\ell +1) }  %%\DN\Dl{\Delta \kl } 
\DN\Di{\Delta _{\infty }  } 
\DN\n{\mathsf{n}}
\DN\nd{\n \Delta }
\DN\hatx{\hat{x}}
%\DR\S{S}
\DN\Sr{\Sit _r}
\DN\Sit{S}
 
\DN\mmF{\mmm _{\mm _{\ell }}}  %\DN\mmF{\mm _{\FFFl }}

\DN\mmFiii{\mmm _{\mm _{\ell ,\iii }}} 
%\DN\mmFiii{\mmm _{\mm _{\ell ,\tinyiii }}}
\DN\mmFiiii{\mmm _{\mm _{\ell , (\ii _1,\ldots,\ii _m)}} }
\DN\mmFiiin{\mmm _{\mm _{\ell , \iii _n}} }

\DN\mmFi{\mmm _{\mm _{\ell , \iiM }}}
\DN\mmFin{\mmm _{\mm _{\ell , \iiM _n }}}
\DN\mmFii{\mmm_{\mm _{\ell , \mathsf{i}}} }

\DN\muFi{\mu _{\FFFl ,\iiM }}
\DN\muFiii{\mu _{\FFFl ,\iii }}
\DN\muk{\mu }
\DN\muz{\mu _{\Dz }}

\DN\mul{\mu |_{\Gl }} %\DN\mul{\mu _{\Dl }} 
\DN\mulsA{\muk ( \mathsf{A} | \Gl  )(\mathsf{s}) }
\DN\mulA{\muk ( \mathsf{A} | \Gl  )} 
\DN\mulsU{\muk ( U | \Gl  )(\mathsf{s}) }

\DN\muls{\muk (\cdot | \Gl  )}
\DN\muF{\mu _{\mathbb{F}}}
\DN\nut{\widetilde{\nu }}
\DN\nul{\nuFl } %\DN\nul{\nu \kl } 
\DN\nuli{\nu _{\ell ,\infty } } 
\DN\nulit{\nut _{\ell ,\infty } } 
\DN\nult{\nut _{\ell }}
\DN\nulc{\check{\nu }_{\ell }}

\DN\SSi{\Ssf _{\infty }}
\DN\Omegat{\widetilde{\Omega}}
\DN\Omegac{\check{\Omega }} \DN\Omegacl{\check{\Omega }_{\ell }}
\DN\tauc{\check{\tau }}
\DN\FF{\mathsf{F}}

\DN\kappal{\kappa _{\ell ,\mathsf{i}}}
\DN\iotal{\iota_{\ell} }
\DN\etai{\eta _{\infty} }
\DN\iotai{\iota _{\infty} }

 \DN\dist{\mathsf{d}}

\DN\0{\rhoFl ^m (\iii )}
\DN\1{\Omega }
\DN\2{\mathrm{Conf}(\1 )}
\DN\3{\overline{\Omega}(\ell )}
\DN\4{\underline{\Omega}(\ell )}
\DN\5{\mu _{\FFFl ,\mathsf{i}}}
\DN\6{\FFFl = \{ \flii \}_{\ii \in\IIIl }}
\DN\7{ \nul \diamond \mmFl }
\DN\8{ ( \nul \diamond \mmF ) \circ \ulabl ^{-1}|_{\Gl } }
\DN\9{\prod_{n=1}^m \Blin }

%======================

\DN\deta{\text{det}_{\alpha}}
\DN\sgn{\alpha ^{n-\nu(\sigma)}}

\DN\p{p}
\DN\q{q}

\DN\mmml{\mmm | _{\Fl}}
\DN\Flm{\Fl ^m}

\DN\iit{\mbox{\boldmath$ i$}}
%\DN\tinyiit{\mbox{\tiny \boldmath$ i$}}
\DN\KRBn{\KR ^{\B,n}}%
\DN\KBn{\K ^{\B,n}}%
\DN\KRBm{\KR ^{\B,m}}%
\DN\KKRBn{\mathsf{L}_{R}^{\B ,n}}%
\DN\LRBmn{\mathsf{L}_{R}^{\B ,m,n}}%
\DN\KBm{\K ^{\B,m}}%
\DN\Bm{\B ^m}%

%\DN\A{\mathfrak{A}}
\DN\A{\mathbb{A}}
\DN\KAn{\K ^{\A , n}}
\DN\KAk{\K ^{\A , k}}
\DN\KAm{\K ^{\A ,m}}
\DN\KRAm{\KR ^{\A ,m}}
\DN\KRAn{\KR ^{\A ,n}}
\DN\KRAk{\KR ^{\A ,k}}

\DN\LR{\mathsf{L}_{R}}
\DN\LRAkn{\LR ^{\A , k , n}}
\DN\s{s}

%===========================

\section{Introduction}\label{s:1} 
Our aim is to introduce tree representations for $ \alpha$-determinantal point processes (also called the $ \alpha$-permanental point processes). 
%Let $ \Sit $ be a locally compact, complete, separable metric space with  metric $ \dist (\cdot, \cdot)$.  
Let $ \Sit $ be a locally compact Hausdorff space with countable basis. 
We equip $ \Sit $ with a Radon measure 
$ \mmm $ such that $ \mmm (\mathcal{O} ) > 0 $ 
for any non-empty open set $ \mathcal{O} $ in $ \Sit $. 
Let $ \Ssf  $ be the configuration space over $ \Sit $ 
(see \eqref{:20a} for definition). 
$ \Ssf  $ is a Polish space equipped with the vague topology.

An $ \alpha$-determinantal point process $ \mu $ on $ \Sit $ 
is a probability measure on $ (\Ssf  ,\mathfrak{B}(\Ssf  ) )$ 
for which the $ m $-point correlation function $ \rho ^m $ 
with respect to $ \mmm $ is given by 
\begin{align}\label{:10a}&
\rho ^m (\mathbf{x}) = \deta [\mathsf{K}(x_i,x_j)]_{i,j=1}^m 
.\end{align}
Here 
$ \map{\K }{\Sit \ts \Sit }{\mathbb{C}}$ is a measurable kernel, 
$ \mathbf{x} = (x_1,\ldots,x_m )$, and 
for $ m \times m $ matrix $ A=(a_{i,j})_{i,j=1}^{m}$ 
\begin{align}\label{:11a}&
\deta A = \sum _{ \sigma \in \mathfrak{S}_m } \alpha ^{m-\nu(\sigma)} \prod _{i=1}^{m}a_{i , \sigma (i)}
,\end{align}
where $ \alpha$ is a real number, 
the summation is taken over the symmetric group $ \mathfrak{S}_m $, 
the set of permutations of $ \{ 1 , 2 , \ldots , m \} $, and 
$ \nu (\sigma) $ is the number of cycles of the permutation $ \sigma$. 
$ \mu$ is said to be $ \alpha$-determinantal point process 
associated with $ (\K , \mmm )$.

The quantity \eqref{:11a} is 
called the $ \alpha$-determinant in \cite{s-t.jfa} and 
also called the $ \alpha$-permanent in \cite{VJ.alpha-det.88, VJ.alpha-det.97}. 
For $ \alpha = -1$, 
$ \det _{-1} A $ is the usual determinant $ \det A$ and 
$ \mu $ is called a determinantal point process (also called a fermion point process). 
For $  \alpha = 1$, 
$ \det _{1} A $ is the permanent $ \text{per} A$ and 
$ \mu $ is called a permanental point process 
(also called a boson point process). 
Letting $ \alpha $ tend to $ 0$, one obtain the Poisson point processes. 
Hence the $ \alpha $-determinantal point process is an one parameter extension of the determinantal point process.

We set $ \KKK f (x)= \int_{\Sit } \K (x,y) f (y) \mmm (dy)$. 
We regard $ \K $ as an operator on $ L^2 (\Sit ,\mmm )$ and denote it by the same symbol. 
We say $ \KKK $ is of locally trace class if 
\begin{align}\label{:111a}&
\KKK _{A} f (x)= \int 1_{A} (x) \K (x,y) 1_{A} (y)f(y)\mmm (dy)
\end{align}
is a trace class operator on $ L^2(\Sit , \mmm )$ 
 for any compact set $ A $. 
Throughout this paper, we assume: 

\smallskip 

\noindent 
\thetag{A1} 
$ \alpha \in \{ \frac{2}{m} ; m \in \N \} \cup \{ \frac{-1}{m}; m \in \N \} $. 
$ \K $ is Hermitian symmetric and 
of locally trace class 
and 
$ \mathrm{Spec}(\K ) \subset [0,\infty ) $
. 
If $ \alpha < 0 $, $ \mathrm{Spec}(\K ) \subset [0,-\frac{1}{\alpha}] $. 

\smallskip 

From \thetag{A1} we deduce that 
 the associated $\alpha$-determinantal point process $ \mu = \mu ^{\K ,\mmm, \alpha }$ 
exists and is unique \cite{s-t.jfa}.

A $ \mmm $-partition $ \Delta =\{ \mathcal{A}_i  \}_{i\in I } $ of $ \Sit $ is a countable collection of disjoint relatively compact, measurable subsets of $ \Sit $ such that $ \cup_i \mathcal{A}_i = \Sit  $ and that 
$ \mmm ( \mathcal{A}_i ) > 0 $ for all $ i \in I $. 
For two partitions $ \Delta =\{ \mathcal{A}_i  \}_{i\in I } $ and 
$ \Gamma=\{ \mathcal{B}_{j}  \}_{j\in J } $, we write 
$ \Delta \prec \Gamma $ if for each $ j \in J $ there exists $ i \in I $ 
such that $ \mathcal{B}_{j}  \subset \mathcal{A}_i $. 
We assume: 

\medskip

\noindent 
\thetag{A2} There exists a sequence of $ \mmm $-partitions 
$ \{ \Delta (\ell ) \}_{\ell \in \N } $ satisfying 
\eqref{:10b}--\eqref{:10d}. 
\begin{align}  \label{:10b}&\quad \quad \quad 
\Dl \prec \Dll \quad \text{ for all } \eN 
,\\ 
\label{:10c}&\quad \quad \quad 
\sigma [\bigcup_{\eN }\Fl ] = \mathfrak{B}(\Sit ) 
,
\\\label{:10d}& \quad \quad \quad 
\# \{ \jj ; \AAAA _{\ell + 1,\jj } \subset \Ali \} = 2 
\text{ for all } \ii \in \Il \text{ and } \ell \in \N 
,\end{align}
where we set 
$ \Dl = \{ \Ali  \}_{i\in \Il } $ and 
$ \Fl := \mathfrak{F}_{\Dl } = \sigma [ \Ali ; i\in \Il  ]$. 

\ms 

Condition \eqref{:10d} is just for simplicity. 
This condition implies that the sequence $ \{ \Delta (\ell ) \}_{\ell \in \N } $ 
has a binary tree-like structure. 
We remark that \thetag{A2} is a mild assumption and, indeed, 
satisfied if $ \Sit $ is an open set in $ \Rd $ 
and $ \mmm $ has positive density with respect to the Lebesgue measure. 

Let $ \Gl $ be the sub-$ \sigma $-field of $ \mathfrak{B} (\mathsf{S})$ given by 
\begin{align}\label{:20e}& 
\Gl  =  \sigma [ \{\sss \in \mathsf{S}; \sss (\Ali ) = n 
\}; i\in \Il , n \in \mathbb{N} ]
.\end{align}
Combining \eqref{:10b} and \eqref{:10c} with \eqref{:20e}, we obtain 
\begin{align}\label{:20f}& 
\Gl \subset \Gll ,\ \quad 
\sigma [\Gl; \ell \in \N ] = \mathfrak{B}(\Ssf  ) 
.\end{align}
Let $ \muls  $ be the regular conditional probability of 
$ \muk $ with respect to $ \Gl  $. 

%We see that the convergence in \eqref{:11b} is stronger than the weak convergence. 

We can naturally regard $ \Dl = \{ \Ali  \}_{i \in \Il } $
 as a discrete, countable set 
with the interpretation that each element $ \Ali  $ is a point. 
Thus, $ \muls $ can be regarded as a point process 
on the discrete set $ \Dl $. 

In \sref{s:2} we introduce a sequence of 
{\it fiber bundle-like sets} $ \IIIl $ ($\ell \in \N )$ 
with base space $ \Dl $ with fiber consisting of a set of binary trees. 
We further expand $ \IIIl $ to $ \Omega (\ell )$ in \eqref{:22t}, which has a fiber 
whose element is a product of a tree $ \ii $ and a component 
$ \Bli $ of partitions. See notation after \tref{l:21}. 

Let $ \mul $ denote the restriction of $ \mu $ on $ \Gl $. 
By construction $ \mul (\mathsf{A}) = \mulA $ for all $ \mathsf{A} \in \Gl $. 
In \tref{l:21} and \tref{l:22}, we construct a lift $ \7 $ of $ \mul $ 
on the fiber bundle $ \Omega (\ell )$ in \eqref{:22t}.

The key point of the construction of the lift $ \7 $ is that we construct 
a consistent family of orthonormal bases $ \6  $ 
in \eqref{:21j} and \eqref{:21k}, and introduce the kernel 
 $ \KFl $ on $ \IIIl $ in \eqref{:21v} such that 
\begin{align}\tag{\ref{:21v}}&
\KFl (\ii ,\jj ) = 
\int_{\Sit \ts \Sit } \K (x,y) 
{\flii  (x) } \fljj  (y) \mmm (dx)\mmm (dy) 
.\end{align}
We shall prove in \lref{l:32} that $ \KFl $ is an $ \alpha$-determinantal kernel on $ \IIIl $, 
and present $ \nuFl $ as the associated $\alpha$-determinantal point process on $ \IIIl $. 
To some extent, $ \nuFl $ is a Fourier transform of $ \mul $ through the orthonormal basis $ \6  $. 
We shall prove in \tref{l:21} that their correlation functions 
$ \rhol ^m $ and $ \rhoFl ^m $ satisfy a kind of Parseval's identity: 
\begin{align}\tag{\ref{:21b}}&
\int_{\mathbb{A}} \rhol ^m(\mathbf{x}) \mmm ^{m}(d\mathbf{x}) = 
\sum_{\iii \in \III_{\ell }(\mathbb{A}) } 
\rhoFl ^m (\iii ) 
,\end{align}
which is a key to construct the lift  $ \7 $.

Vere-Jones \cite{VJ.alpha-det.88, VJ.alpha-det.97} 
introduced $ \alpha$-permanent (we call it $ \alpha$-determinant as refereed in  \cite{s-t.jfa})
as the coefficients which arise in 
expanding fractional powers of the characteristic polynomial of a matrix. 
Shirai-Takahashi \cite{s-t.jfa} 
introduced the $ \alpha$-determinantal point processes. 
Their correlation functions are given by $ \alpha$-determinants of a kernel function. 
In the case $ \alpha=-1$, the associated point process is the determinantal point processes \cite{hkpv.gaf, RL.03, RL.14, sosh, s-t.jfa, s-t.aop}. 
The condition \thetag{A1} is a part sufficient condition for the existence and uniqueness of $ \alpha$-determinantal point process in \cite{s-t.jfa}. 

In \cite{o-o.tail}, we introduced the tree representations for determinantal point processes on a continuum space under the assumption \thetag{A1} in the case $ \alpha=-1$ 
and 
proved tail triviality by applying it. 
In this paper, we prove that the tree representations work for the $ \alpha$-determinantal point processes. 
Most statements in this paper are then the same as \cite{o-o.tail} except for the range of $ \alpha$. 
In particular, \lref{l:31} and \lref{l:33} correspond to Lemma 1 and Lemma 3 in \cite{o-o.tail}, respectively. 

The key idea is that $ \KFl $ in \eqref{:21v} is given
by a unitary operator 
$ U :  L^2 (\Sit , \mmm ) \rightarrow L^2(\IIIl , \lambdal  )$ such that $ \K = U \KFl U^{-1} $. 
Hence $ \KFl $ has the same spectrum of $ \K $ and satisfies \thetag{A1}.

The organization of the paper is as follows. 
In \sref{s:2}, we give definitions and concepts and state the main theorems (Theorems \ref{l:21}--\ref{l:23}). We give tree representations of $ \mu $. 
In \sref{s:3}, we prove \tref{l:21}. 
In \sref{s:4}, we prove \tref{l:22} and \tref{l:23}. 

\section{Tree representations}\label{s:2}
%Let $ ( \Sit ,\mmm ) $ be as in \sref{s:1}. 

In this section, we recall various essentials and present the main theorems 
\tref{l:21}--\tref{l:23}. 

A configuration space $ \Ssf  $ over $ \Sit $ 
is a set consisting of configurations on $ \Sit $ such that 
\begin{align}& \label{:20a}
\Ssf  = \{ \sss \, ;\,   \sss =\sum_i \delta _{s_i},\, 
 s_i \in \Sit ,\, 
\sss (K) < \infty \text{ for any compact } K \} 
.\end{align}
A probability measure $ \mu $ on $ ( \Ssf  , \mathfrak{B}(\Ssf  ) )$ 
is called a point process, also called random point field. 
 A symmetric function $\rho ^m $ on $\Sit ^m$ is called 
the $ m $-point correlation function of a point process $ \mu $ 
with respect to a Radon measure $ \mmm $ if it satisfies 
\begin{align} \label{:20b}
\int_{\Ssf  } 
\prod_{i=1}^j \frac{\sss (A_i)!}{(\sss (A_i) - k_i)!} 
\mu(d\sss ) 
&= \int_{A_1^{k_1} \times \cdots \times A_j^{k_j}} 
\rho ^m (\mathbf{x}) \mmm ^m(d\mathbf{x})
.\end{align}
Here $A_1, \dots, A_j \in \mathfrak{B}(\Sit )$  are disjoint and 
$k_1,\dots, k_j \in \N $ such that $k_1+\cdots + k_j = m $. 
If $\sss (A_i) - k_i \le 0$, we set 
${\sss (A_i)!}/{(\sss (A_i) -k_i)!}=0$.

Let $ \Dl =\{ \Ali  \}_{i\in \Il  } $ be as in \thetag{A2}, where 
 $ \ell \in \N $. 
We set $ \Delta =  \{ \mathcal{A}_{\ii }  \}_{\ii \in I } $ such that 
 \begin{align}\notag &% \label{:20d}&
 \Delta  = \Done , \quad \mathcal{A}_{\ii }=\mathcal{A}_{1,\ii  }  \quad I = I (1)
 .\end{align}
In consequence of  \eqref{:10d}, we assume 
without loss of generality that each element $ \ii  $ of 
the parameter set $ \Il $ is of the form 
\begin{align}\label{:20g}&
\Il = \I \ts \{ 0,1 \}^{\ell -1}
.\end{align}
That is, each $ \ii \in \Il $ is of the form 
$ \ii  = (j_1,\ldots,j_{\ell }) \in \I \ts \{ 0,1 \}^{\ell -1} $. 
We take a label $ \ii \in \cup_{\ell = 1}^{\infty} \Il $ in such a way that, 
for $ \ell < \ell ' $, $ \ii \in \Il $, and $ \ii ' \in I (\ell ' )$, 
\begin{align}\notag &% \label{:20h}&
\Ali  \supset \AAAA _{\ell ', \ii '}  
\Leftrightarrow 
\ii =(j_1,\ldots,j_{\ell }) \text{ and }
\ii '=(j_1,\ldots,j_{\ell },\ldots,j_{\ell '})
.\end{align}
We denote by $ \wIII $ the set of all such parameters:
\begin{align}\label{:20i}&
\wIII  = \sumii{\ell } \Il = \sumii{\ell }  \I  \ts \{ 0,1 \}^{\ell -1} 
.\end{align}
We can regard $ \wIII $ as 
a collection of binary trees and $ \I  $ is the set of their roots.

For $ \ii  = (j_1,\ldots,j_{\ell }) \in \wIII $, we set 
$ \mathrm{rank}(\ii ) = \ell $. 
For $ \ii $ with $ \mathrm{rank}(\ii ) = \ell $, 
we set 
\begin{align}\label{:20j} &\Bi  =
\begin{cases}
 \mathcal{A}_{1,\ii } & \ell = 1,\\
\mathcal{A}_{\ell -1, \ii ^-}  
&  \ell \ge 2
,\end{cases}
\end{align}
where $ \ii ^- = (j_1,\ldots,j_{\ell -1})$ for 
$ \ii = (j_1,\ldots, j_{\ell })\in \Il $. 
Let $ \III \subset \wIII $ such that 
\begin{align}\label{:20k}&
\III = \I +  \sum_{\ell = 2}^{\infty}
\{ \ii  \in \Il ; j_{\ell }=0   \}
,\end{align}
where $ \ii = (j_1,\ldots,j_{\ell }) \in \Il $. 

Let $ \mathbb{F} = \{ \fii \}_{\ii \in\mathbb{I} } $ 
be an orthonormal basis of $ L^2(\Sit ,\mmm ) $ satisfying 
\begin{align}
& && \label{:20l}
\sigma [\fii ; \ii \in \III ,\, \mathrm{rank} (\ii ) = \ell ] = \Fl 
&&\text{ for each $ \ell \in \N $} 
,\\ \label{:20m}&&&
\mathrm{supp}(\fii ) 
= 
\Bi  
&& \text{ for each $ \ii \in \III $}
,\\
\label{:20n}&&&
\fii  (x) = 
1_{\Ai }(x)/\sqrt{\mmm (\Ai )}
&&\text{ for $ \mathrm{rank} (\ii ) = 1$}
.\end{align}
For a given sequence of $ \mmm $-partitions 
satisfying \thetag{A2}, such an orthonormal basis exists. 
We present here an example. 

\begin{exa}[Haar functions] \label{d:22} 
Typically we can take $ \Sit = \R $, $ \mmm (dx)= dx $, and $ \I  = \mathbb{Z}$. 
For $ \ii = (j_1,\ldots,j_{\ell }) \in \Il $, we set 
$ \Jonei = j_1 $ and, for $ \ell \ge 2 $, 
\begin{align}\label{:20q}&
\Jli  = j_1+\sum_{n=1}^{\ell - 1} \frac{j_n}{2^{n}}
.\end{align}
We take $ \Ali = [\Jli , \Jli  + 2^{-\ell +1})$.

Let $ \ii = (j_1,\ldots,j_{\ell }) \in \III $. We set for, $ \ell = 1 $ and $ \ii = (j_1)$, 
\begin{align}\notag & %\label{:20r}&
f_{\ii } (x) = 
1_{[j_1,j_1+1)}  (x) %\quad (j_1\in\Z )
\end{align}
and, for $ \ell \ge 2 $ and $  \ii = (j_1,\ldots,j_{\ell }) \in \III $, 
\begin{align}\notag & %\label{:20r}&
f_{\ii } (x) = 
2^{(\ell -1)/2 } \{ 1_{[\Jli , \Jli  + 2^{-\ell +1})}(x) 
- 
1_{[\Jli + 2^{-\ell +1}, \Jli  + 2^{-\ell +2 })}(x) \}  
.\end{align}
We can easily see that $ \{ \fii  \}_{\ii \in \III } $ is an orthonormal basis of 
$ L^2(\R , dx )$. 
We remark that $ j_{\ell } = 0 $ because $  \ii = (j_1,\ldots,j_{\ell }) \in \III $ as we set in \eqref{:20k}.  
\end{exa}

We next introduce the $ \ell $-shift of above objects 
such as $ \III $, $ \Bi $, and $ \mathbb{F} = \{ \fii \}_{\ii \in\mathbb{I} } $. 
Let $ \wIII (1) = \wIII $ and, for $ \ell \ge 2 $, we set 
\begin{align}\label{:21g}
\quad \quad \wIIIl := &
 \sumii{\rfrak } 
\Il \ts \{ 0,1 \}^{\rfrak -1 } 
,\end{align}
where $ \Il =  \I \ts \{ 0,1 \}^{\ell -1}$ is as in \eqref{:20g}.  
For $ \ell , \rfrak \in  \N  $, we set 
$ \map{\theta _{\ell -1 ,\rfrak }}{\wIII }{ \wIIIl }$ 
such that 
$ \theta _{0,\rfrak } = \mathrm{id}$ ($ \ell = 1$) and, for $ \ell \ge 2$,   
\begin{align}\label{:21h}&
\theta_{\ell -1,\rfrak } ((j_1,\ldots,j_{\ell + \rfrak -1 })) 
= ( \mathbf{j}_{\ell }, j_{\ell +1 }, \ldots,j_{\ell + \rfrak -1 }) 
\in \Il  \ts \{ 0,1 \}^{\rfrak -1 } 
,\end{align}
where $ \bm{j}_{\ell }=(j_1,\ldots,j_{\ell }) \in \Il $. 
For $ \ell = 1$, we set $ \III (1) = \III $. For $ \ell \ge 2 $, we set 
\begin{align}\label{:21i}&
\IIIl  = \Il + 
\sum_{\rfrak = 2}^{\infty}  \theta_{\ell -1 ,\rfrak }  (\III )
.\end{align}

We set $ \mathrm{rank}(\ii ) = r $ for $ \ii \in  \Il  \ts \{ 0,1 \}^{\rfrak -1 } $. 
By construction $ \mathrm{rank}(\ii ) = r $ 
for $ \ii \in \theta_{\ell -1,\rfrak } (\wIII )$. 
Let $ \6  $ such that, 
for $ \rfrak = \mathrm{rank} (\ii ) $,  
\begin{align}\label{:21j}&&&\quad \quad 
\flii  (x) = 
1_{\Ali }(x)/\sqrt{\mmm (\Ali )}
&&\text{ for } \rfrak = 1
,\\
&&& \label{:21k} \quad \quad  \flii (x) = 
f_{ \theta_{\ell -1 ,\rfrak } ^{-1}(i)} (x) 
&&
\text{ for } \rfrak  \ge 2 
,\end{align}
where $ \Dl = \{ \Ali  \}_{i\in\Il } $ is given in \thetag{A2}. 
Then $ \6  $ 
is an orthonormal basis of $ L^2(\Sit ,\mmm ) $. 
This follows from assumptions \eqref{:21j} and \eqref{:21k} and 
the fact that $ \FFF = \{ \fii \}_{\ii \in\mathbb{I} } $ 
is an orthonormal basis.

\begin{rem}\label{r:21} \thetag{1} 
We note that 
$ \flii \in \FFFl $ is a newly defined function 
if $ \mathrm{rank}(\ii ) = 1 $, whereas 
$ \flii \in \FFFl $ is an element of $ \FFF $ if 
$ \mathrm{rank}(\ii ) \ge 2 $. 
In particular, we see that 
\begin{align}\label{:21l}&
\{ \flii  \}_{\ii \in \IIIl , \, \mathrm{rank}(\ii )\ge 2} 
\subset 
\{ \fii  \}_{\ii \in \III , \, \mathrm{rank}(\ii )\ge 2} 
.\end{align}
\thetag{2} 
Let 
$ \jj = (j_1,\ldots,j_{\ell + \rfrak -1})\in\III $ 
and 
$ \ii  = (\bm{j}_{\ell }, j_{ \ell + 1 },\ldots,j_{\ell + \rfrak -1}) \in \IIIl $. 
%Then $ \mathrm{rank} (\ii ) = \rfrak $. 
Then 
$$ \jj =  \theta_{\ell -1 ,\rfrak } ^{-1}(\ii )
.$$
Furthermore, $ \flii  \in\FFFl  $ and $ \fjj \in  \FFF $ satisfy $ \flii = \fjj   $ 
for $ r=\mathrm{rank} (\ii ) \ge 2 $. 
\\
\thetag{3} By construction, we see that 
\begin{align} 
& && \label{:21s}
\sigma [\flii ; \ \ii \in \IIIl , \, \mathrm{rank} (\ii ) = r ] = 
\mathfrak{F}_{\ell -1 + r}  
&&\text{ for each $ \ell , r \in \N $} 
,\\ \label{:21t}&&&
\mathrm{supp}(\flii ) 
 = \Bli 
&&
\text{ for all } \ii \in \IIIl 
,\end{align}
where we set, for $  \jj = \theta_{\ell -1 ,\rfrak } ^{-1}(\ii )$ 
such that $ \mathrm{rank}(\ii ) = r $, 
\begin{align}\label{:21u}&
\Bli  = \Bj  
.\end{align}
\end{rem}

Using the orthonormal basis 
$ \6  $, 
we set $ \KFl $ on $ \IIIl $ by 
\begin{align}\label{:21v}&
\KFl (\ii ,\jj ) = 
\int_{\Sit \ts \Sit } \K (x,y) 
{\flii  (x) } \fljj  (y) \mmm (dx)\mmm (dy) 
.\end{align}
Let $ \lambdal $ be the counting measure on $ \IIIl $. 
We shall prove in \lref{l:32} that $ (\KFl , \lambdal )$ satisfies \thetag{A1}. 
Hence we obtain the associated $\alpha$-determinantal point process $ \nuFl $ on $ \IIIl $ 
from general theory \cite{s-t.jfa}.

For $ \iiM  \in \IIIl $, let $ \mmFi  (dx ) $ be 
the probability measure on $ \Sit $ such that 
\begin{align}\label{:21w}&
\mmFi  (dx) =  | \flii (x)|^2 \mmm (dx) 
.\end{align}
For $ \iii = (\ii _n)_{n=1}^m \in \IIIl ^m $ and 
$ \mathbf{x} = (x_n)_{n=1}^m  $, where $ m \in \N \cup \{ \infty  \} $, we set 
\begin{align}\label{:21x}&
\mmFiii  (d\mathbf{x}) = \prod_{n=1}^m 
 |f_{\ell ,\ii  _n }(x_n)|^2 \mmm (dx_n) 
.\end{align}
By \eqref{:21k} $ \mmFiii $ is a probability measure on $ \Sit ^m $. 
By \eqref{:21t}, we have 
\begin{align}& \label{:21y}
\mmFiii   ( \9 ) = 1 
.\end{align}

Let $ \Flm = \sigma [ \AAAA _{\l ,\ii _1} \times \cdots \times \AAAA _{\l ,\ii _m}; \ii _{n} \in \Il , n=1,\ldots,m]$. 
Let $ \Gl $ be the sub-$ \sigma $-field  as in \eqref{:20e}.  
An $ \Flm $-measurable symmetric function $\rhol ^m $ on $\Sit ^m$ 
is called 
the $ m $-point correlation function of $ \mul $ 
with respect to $  \mmm $ if it satisfies 
\begin{align} \label{:21z}
\int_{\Ssf  } 
\prod_{i=1}^j \frac{\sss (A_i)!}{(\sss (A_i) - k_i)!} 
\mu(d\sss ) 
&= \int_{A_1^{k_1} \times \cdots \times A_j^{k_j}} 
\rhol ^m (\mathbf{x}) \mmm ^m(d\mathbf{x})
.\end{align}
Here $A_1, \dots, A_j \in \Fl $  are disjoint and 
$k_1,\dots, k_j \in \N $ such that $k_1+\cdots + k_j = m $. 
If $\sss (A_i) - k_i \le 0$, we set 
${\sss (A_i)!}/{(\sss (A_i) -k_i)!}=0$.

Let $ \nul $ be the $ \alpha$-determinantal point process associated to $ (\KFl , \lambdal )$ 
as before. 
Let $ \rhol ^m $ and $ \rhoFl ^m $ be 
the $ m $-point correlation functions of $ \mul $ and 
$ \nuFl $ with respect to $ \mmm $ and $ \lambdal $, respectively. 
We now state one of our main theorems:

\begin{thm} \label{l:21} 
Let $ \III_{\ell }(\AAAA ) = \{ \ii \in \IIIl \, ;\,  \Bli \subset \AAAA \} $. 
%Let $ \ell , m \in \N $ and $ \mathbb{A}=  \AAAA _1\ts \cdots \ts \AAAA _m $.  
For $ \mathbb{A}=  \AAAA _1\ts \cdots \ts \AAAA _m$, we set 
\begin{align}\label{:21a}&
\III_{\ell }(\mathbb{A}) = 
\III_{\ell }(\AAAA _1)\ts \cdots \ts \III_{\ell }(\AAAA _m ) 
.\end{align}
Assume that 
$\AAAA _n \in \Dl  $ for all $ n = 1,\ldots,m $. 
Then
\begin{align}\label{:21b}&
\int_{\mathbb{A}} \rhol ^m(\mathbf{x}) \mmm ^{m}(d\mathbf{x}) = 
\sum_{\iii \in \III_{\ell }(\mathbb{A}) } 
\rhoFl ^m (\iii ) 
.\end{align}
\end{thm}

Let $ \mathsf{I}(\ell ) $ be the configuration space over $ \IIIl $. 
%We write $ \ii \in \mathsf{I}(\ell )$ if $ \mathsf{i}(\{ i \} ) \geq 1 $. 
%Each $ \mathsf{i}=\sum_{i\in\mathsf{i}} \ii (\{ i \}) \delta_{i}\in  \mathsf{I}(\ell ) $ can be regarded as a subset of $ \IIIl $ by the correspondence of $ \mathsf{i} $ to $ \{ i \}_{i\in\mathsf{i}} $. 
Let 
\begin{align}& \label{:22t}
 \1 (\ell )  := 
\bigcup_{i \in \IIIl } \{ i \} \ts \Bli  
.\end{align}
Let $ \4   $ be the configuration space over $ \1 (\ell )$. 
Then by definition each element $ \omega \in \4   $ is of the form 
$ \omega = \sum_{n} \delta_{(\ii _{n} , s_{n})}$ 
such that $ s_{n} \in \Blin $. Hence 
\begin{align} \label{:22u}
\4 &  = \{ \omega = \sum_{n} \delta_{(\ii _{n} , s_{n})} 
\, ;\, \mathsf{i} =\sum_{n} \delta_{\ii _n} \in \IIl , \, 
 s_{n} \in \Blin     \} 
.\end{align}
We exclude the zero measure from $ \4 $. 

Let $ \mmFi $ be as in \eqref{:21w}. We set  
\begin{align}
\label{:22v}
\mmFl = & 
 \prod_{{ \ii \in \IIIl  }}  \mmFi  
,\quad 
\mmFii = 
 \prod_{n}  \mmFin  
.\end{align}

\begin{rem}\label{r:22}
\thetag{1}
A configuration $ \mathsf{i} \in \IIl $ can be represented as $ \mathsf{i} = \sum _{n} \delta_{\ii _{n}} $ and this may have multiple points. 

\thetag{2}
Let $ \mathsf{i} \in \IIl $. 
Suppose that for some $ m \in \N \cup \{\infty\}$, 
$ \mathsf{i}$ has plural representations such as 
\begin{align}\notag
\mathsf{i} = \sum_{n=1}^m \delta_{i_n} = \sum_{n=1}^m \delta_{j_n}
.\end{align}
Then $ \prod_{n=1}^m \Blin $ and $ \prod_{n=1}^m \Bljn $ can be different subsets of $ \Sit ^m$. 
However, the product probability spaces 
$ (\prod_{n=1}^{m} \Blin , \mmFii )$ and $ (\prod_{n=1}^{m} \Bljn , \mmFii )$ 
are the same under the identification such that 
\begin{align}\notag
\prod _{n=1}^m \Blin \ni (x_n)_{n=1}^m 
\mapsto 
(x_{\sigma(n)})_{n=1}^m \in \prod _{n=1}^m \Bljn 
.\end{align}
Here, $ \sigma $ is the permutation such that $ i_{\sigma(n)}=j_n$. 
They do not depend on the representations of $ \mathsf{i}$ under this identification. 
\end{rem}

We set $ \map{ \iotal  }{ \4 }{ \mathsf{I}(\ell ) }$ 
such that $ \sum_{n}  \delta_{(\ii _n , s_n )} \mapsto \sum_{n} \delta_{i_n} $. 
For $ \mathsf{i} \in \IIl $, let 
\begin{align}\notag 
\map{ \kappal }{  \{ \omega \in \4 ; \iota _{\ell} ( \omega ) = \mathsf{i} \} }{   \prod_{n} \Blin }
\end{align}
such that $ \sum_{n}  \delta_{(\ii _n , s_n )} \mapsto (s_n)
$. 
Let $ \7 $ be the probability measure  on $ \4 $ given by 
the disintegration made of 
\begin{align} \label{:22w} & \quad \quad 
(\7 ) \circ \iotal ^{-1} (d\mathsf{i}) = \nul (d\mathsf{i})
,\\\label{:22x}&\quad \quad 
\7 ( \kappal (\omega ) \in d\mathbf{s}  | \iotal (\omega ) = \mathsf{i}) = 
   \mmFii (d\mathbf{s})  
,\quad \mathbf{s} =(s_n)_{}
\text{ for } \mathsf{i} = \sum_{n} \delta_{i_n}
.\end{align}

\begin{rem}\label{r:23} \thetag{1} 
We can naturally regard the probability measures in \eqref{:22x} 
as a point process on $ \prod_{n} \Blin $ 
supported on the set of configurations with exactly one particle configuration 
$ \mathsf{s}= \delta_{\mathbf{s}}$ on 
$ \prod_{n} \Blin $, that is, 
$ \mathbf{s}=(s_n)_{}$ is such that  
$ s_n \in \Blin $. 
\\\thetag{2} 
We can regard $ \7 $ as a marked point process as follows: The configuration $ \mathsf{i}$ is distributed according to $ \nul $, while the marks are independent and for each $ \mathsf{i}$ the mark $ \mathbf{s}$ 
is distributed according to $ \mmFii $. Thus the space of marks depends on $ \mathsf{i}$. 

\end{rem}

\begin{thm}	\label{l:22} 
Let $ \map{\ulabl  }{\4 }{\Ssf  }$ be such that 
$ \ulabl (\omega ) = \sum_{n} \delta_{s_{n}}$, 
where $ \omega = \sum_{n} \delta_{(\ii _n, s_{n} )} $. 
Then 
\begin{align}\label{:22a}&
\mul  = (\7 )\circ  \ulabl  ^{-1} |_{ \Gl } 
.\end{align}
\end{thm}

\begin{rem}\label{r:4} 
\tref{l:22} implies that $ \7 $ is a {\em lift} of $\mul $ onto $ \4 $. 
We can naturally regard  $ \wIIIl  $ as binary trees. 
Hence we call $ \7 $ a tree representation of $ \mu $ of level $ \ell $. 
% \\\thetag{3} 
% Note that $ \nul $ is a $\alpha$-determinantal point process 
% on the discrete space $ \Omega (\ell )$. 
% It enjoys many well-behaved properties of 
% $\alpha$-determinantal point processes on discrete spaces. 
% $ \mmFl $ is a product probability measures. 
% Both measures have thus very nice properties, 
% which are  inherited by $ \7 $. 
% The typical example of this is  tail triviality. 
% Indeed, both $ \nul $ and $ \mmFl $ are tail trivial. 
%Using this, we shall prove tail triviality of $ \7 $ in  \tref{l:24}. 
\end{rem}

%We suppress the domain from the notation of $ \ulabli $. 
We present a decomposition of $ \mul $, 
which follows from \tref{l:22} immediately. 
Let $ \mmFii ^{\ulab } = \mmFii \circ \ulabi ^{-1}$ 
%for $ \mathsf{i} = \sum_{n} \delta_{i_n}$ 
, where 
$ \map{\ulabi }{ \prod_{n} \Blin  }{\Ssf  }$ 
is the unlabel map 
such that 
\begin{align}\label{:22b}&
\ulabi ((s_{n})) = 
\sum_{n} \delta_{s_{n}}
. \end{align}
\begin{thm}	\label{l:23} 
For each $ \mathsf{A}\in \Gl $, 
\begin{align}\label{:23b}& \quad \quad 
\mu (\mathsf{A}) = 
\int_{\IIl }\nuFl (d\mathsf{i}) \,  \mmFii ^{\ulab }    (\mathsf{A}) 
.\end{align}
\end{thm}

We remark that $ \mul $ is {\em not} an $ \alpha$-determinantal point process. 
Hence we exploit $ \7 $ instead of $ \mul $. 
As we have seen in \tref{l:22},  
$  \7 $ is a lift of $ \mul $ in the sense of \eqref{:22a}, 
from which we can deduce nice properties of $ \mul $. 
Indeed, an application of \tref{l:22} 
is  tail triviality of $ \mu $ in the case $ \alpha =-1$ \cite{o-o.tail}. 

\section{Proof of \tref{l:21}}\label{s:3}
The purpose of this section is to prove \tref{l:21}. 
In \lref{l:31}, we present 
a kind of Parseval's identity of kernels $ \K $ and $ \KFl $ 
using the orthonormal basis 
$ \FFFl  $, 
where $ \KFl $ is the kernel given by \eqref{:21v} and 
$ \FFFl  $ is as in \eqref{:21j} and \eqref{:21k}. 
In \lref{l:32}, we prove $ (\KFl , \lambdal )$ is a determinantal kernel and 
the associated $\alpha$-determinantal point process $ \nul $ exists. 
We will lift the Parseval's identity between $ \K $ and $ \KFl $ 
to that of correlation functions of $ \mul $ and $ \nul $ in \tref{l:21}.  

% Let $ (\mu ,\K ,\mmm )$ be as in \sref{s:2}, 
% that is, $ \mu $ is the $ (\K , \mmm )$-determinantal point process. 
% and 
% 
% $ \6  $ is the orthonormal basis of $ L^2(\Sit ,\mmm )$
%an isomorphism between 
%$ L^2(\Sit ,\mmm )$ and $ L^2 (\FFFl ,\lambdal )$ 

By definition $ \6 $ satisfies 
\begin{align}\label{:30a}&
\int_{\Sit } |\flii (x)|^2 \mmm (dx)= 1 \quad 
\text{ for all } \ii \in\IIIl 
,\\\label{:30b}&
\int_{\Sit } \flii (x) {\fljj (x)}  \mmm (dx)= 0 \quad 
\text{ for all } \ii \not=\jj \in \IIIl 
.\end{align}
% with $ \mathbb{F} = \FFFl $. 
\begin{lmm} \label{l:31} 
\thetag{1} Let $ P(x) = \sum_i \p  (\ii ) f _{\ell ,i} (x) $ and 
$ Q(y) = \sum_j \q  (\jj ) f _{\ell ,j} (y) $. 
Suppose that the supports of $ \p  $ and $ \q  $ are finite sets. 
Then 
\begin{align}\label{:31z}&
\int_{\Sit \ts \Sit }   \K (x,y) \overline{P (x)}Q (y) \mmm (dx) \mmm (dy)  = 
\sum_{i,j} \KFl (i,j) \overline{\p  (\ii )}\q  (\jj ) 
.\end{align}
\thetag{2} We have an expansion of 
$ \K $ in $ L^2_{\mathrm{loc}}(\Sit \ts \Sit ,\mmm \ts \mmm )$ such that 
\begin{align}\label{:31a}&
\K (x,y) = \sum_{\ii ,\jj \in\IIIl }
\KFl (\ii ,\jj ) 
\flii  (x) 
{\fljj  (y)} 
.\end{align}
\thetag{3} 
Let 
$ \III (\ell ; R ) = \{ i \in \IIIl ; \mathrm{rank}(i) \le R \}$, 
where $ \mathrm{rank}(i)$ is defined before \eqref{:21j}. 
Let 
\begin{align}\label{:31y}&
\KR (x,y)= 
\sum_{\ii , \jj \in \III (\ell ; R )  } 
 \KFl (\ii ,\jj ) 
\flii (x) 
{\fljj (y)}
.\end{align}
We set $\mathbb{A}=  \AAAA _1\ts \cdots \ts \AAAA _m $. 
Assume that $\AAAA _n \in \Dl $ for $  n = 1,\ldots,m $. 
Then for $ \sigma \in \mathfrak{S}_m$, 
\begin{align}\label{:31x}&
\limi{R}
\int_{\mathbb{A}  } 
\prod_{n=1}^m 
\KR (x_n,x_{\sigma(n)})
\mmm ^{m}(d\mathbf{x}) 
=
\int_{\mathbb{A}  } 
\prod_{n=1}^m 
\K (x_n,x_{\sigma(n)})
\mmm ^{m}(d\mathbf{x}) 
.\end{align}
%where $ \III_{\ell }(\mathbb{A}) $ is defined before \eqref{:21a}. 
\end{lmm}
\begin{proof}
From \eqref{:21v} we deduce that  
\begin{align}\label{:31b}&
\int_{\Sit \ts \Sit }   \K (x,y) 
\overline{P (x)}Q (y) \mmm (dx) \mmm (dy) 
\\ \notag =& 
\int_{\Sit \ts \Sit }  \K (x,y) 
\overline{\sum_i \p  (\ii ) f _{\ell ,i} (x) } 
\sum_j \q  (\jj ) f _{\ell ,j} (y) \mmm (dx) \mmm (dy) 
\\ \notag =&
\sum_{i,j} \int_{\Sit \ts \Sit} 
 \K (x,y)  {\flii (x)}  \fljj (y) \mmm (dx) \mmm (dy) 
\overline{\p  (\ii )}\q  (\jj ) 
\\ \notag =&
\sum_{i,j} \KFl (i,j) \overline{\p  (\ii )}\q  (\jj ) 
.\end{align}
This yields \eqref{:31z}. We have thus proved \thetag{1}. 
% We next prove \thetag{2}. 
By a direct calculation, we have 
\begin{align}\label{:31c}&
 \int_{\Sit }  \overline{P(x)} \flii (x)  \mmm (dx) = 
\int_{\Sit }  
\sum_{\jj } \overline{ \p (\jj ) }f _{\ell ,\jj } (x) \flii (x)  \mmm (dx) = 
\overline{\p (\ii ) } 
,\\ \notag &
 \int_{\Sit }  Q(y) {\fljj (y)}  \mmm (dy) = 
\int_{\Sit }  \sum_{\ii } \q  (\ii ) f _{\ell ,\ii } (y) {\fljj (y)}  \mmm (dy) = \q (\jj ) 
.\end{align}
Combining \eqref{:31b} and \eqref{:31c} yields 
\begin{align}\notag %\label{:31d}&
&
\int_{\Sit \ts \Sit }   \K (x,y) \overline{P (x)}Q (y) \mmm (dx) \mmm (dy) = 
\\ \notag &
\int_{\Sit \ts \Sit }  \sum_{i,j} \KFl (i,j) 
\flii (x) {\fljj (y)} \overline{P (x)}Q (y) \mmm (dx) \mmm (dy) 
.\end{align}
This implies \eqref{:31a}.

Without loss of generality, we can assume $ \sigma$ is a cyclic permutation. 
We prove \eqref{:31x}  only for  $ \sigma = (1,2,\ldots,m) $. 
Let $ \AAAA _{n} \in \Dl$ for $ n \geq 0 $. 
Let $ \A = \A (m)= \AAAA _{0} \times\cdots\times\AAAA _{m} $.  
For $ 0 \leq n \leq m $, we set 
\begin{align}\label{:31d}&
\KAn (x,y) =
\int _{\AAAA _{1} \times \cdots \times \AAAA _{n-1}} 
\prod _{p=1}^{n}\K (x_{p-1},x_{p})\mmm(dx_1)\cdots\mmm(dx_{n-1})
,\\\label{:31f}&
\KRAn (x,y) =
\int _{\AAAA _{m-(n-1)} \times \cdots \times \AAAA _{m-1}} 
\prod _{p=1}^{n}\KR (x_{p-1},x_{p})\mmm(dx_{1})\cdots\mmm(dx_{n})
.\end{align}
where $ x_0 =x $, $ x_n =y $, $ \K ^{\A,0}(x,y) = \KR ^{\A ,0}(x,y) = \delta_{x}(y) $, $ \K ^{\A,1}(x,y)=\K (x,y) $, and $ \KR ^{\A,1}(x,y)=\KR (x,y)$. 
By assumption $ \K $ is a trace class operator on $ L^2 (\B , \mmm)$ for a relatively compact set $ \B $ such that $ \bigcup_{p=1}^m \AAAA _p \subset \B $. 
Then $ \KAn $ is also a trace class operator on $ L^2 (\B , \mmm)$ for each $ n \in \{1,\ldots,m\} $. 
In particular, $ \KAn $ is a Hilbert-Schmidt operator on $ L^2 (\B , \mmm)$ and satisfies 
\begin{align}\label{:31e}&
\int _{\B ^2} |\KAn (x,y)|^2 \mmm(dx)\mmm(dy) < \infty
.\end{align}
We set for $ k,n \geq 0 $ such that $ k+n =m$, 
\begin{align}\label{:31g}&
\LRAkn  (x,y) =
\int _{\AAAA _{k}} 
\KAk (x,z) \KRAn (z,y)\mmm(dz)
.\end{align}
We shall prove the following by induction for $ m $ : 
for  all $ k , n \geq 0 $ such that $ k+n = m $ 
and for any $ \mathbb{A}=\AAAA _0 \times \cdots \times \AAAA _m$ 
such that $ \AAAA _p \in \Dl $ for $ p= 0 ,\ldots, m$
\begin{align}\label{:31h}&
\lim_{R \to \infty}
\int_{\AAAA _0 \times \AAAA _m}|\LR ^{\A ,k ,n} (x,y) - \K ^{\A,m} (x,y)|^2 \mmm (dx)\mmm(dy) = 0
,\\ \label{:31i}&
\sup_{R} \int _{\AAAA _0 \times \AAAA _m}|\LR ^{\A ,k ,n} (x,y) |^2 \mmm(dx)\mmm(dy) < \infty
.\end{align}

Let $ m=1$. 
For $ (k,n) = (0,1)$, \lref{l:31} \thetag{2} implies \eqref{:31h} and \eqref{:31i}. 
For $ (k,n) = (1,0)$,  $ \LR ^{\A ,1 ,0}(x,y) =\K ^{\A , 1}(x,y) $ by the definition in \eqref{:31g}. 
Then \eqref{:31h} and \eqref{:31i} hold for $ (k,n)=(1,0) $. 
Hence \eqref{:31h} and \eqref{:31i} holds for $ m=1$. 

Suppose \eqref{:31h} and \eqref{:31i} hold for $1,\ldots, m-1 $. 
Let $ k + n =m-1$ and $ \mathbb{A}=\AAAA_0 \times \cdots \times \AAAA _m$. 
By a straightforward calculation, 
\begin{align}\notag& %\label{:31j}&
%|
\LR  ^{\A ,k, n+1 }(x,y)-\LR ^{\A ,k+1,n} (x,y) 
%|^2
\\ \notag =&
%\Bigl|
\int _{\AAAA _k} \K ^{\A , k }(x,z) \KR ^{\A , n+1}(z,y)\mmm(dz)
-
\int _{\AAAA _{k+1}} \K ^{\A , k+1 }(x,w) \KR ^{\A , n }(w,y) \mmm(dw)
%\Bigr|^2
\\ \notag =&
\int _{\AAAA _k \times \AAAA _{k+1}}  
\K ^{\A , k }(x,z) 
\KR (z,w) 
\KR ^{\A , n}(w,y)
%\mmm(dz)\mmm(dw)
%\\ \notag &
-
%\int _{\AAAA _{k} \AAAA _{k+1}} 
\K ^{\A , k }(x,z) \K (z,w) \KR ^{\A , n }(w,y) \mmm(dz)\mmm(dw)
\\ \notag =&
\int _{\AAAA _k \times \AAAA _{k+1}} 
\K ^{\A ,k}(x,z) \KR ^{\A , n}(w,y) 
\Bigl(
\KR (z,w)-\K (z,w)
\Bigr)
\mmm(dz)\mmm(dw)
.\end{align}
By the Schwartz inequality for the last term, we have 
\begin{align}\notag&
\Bigl|
\LR  ^{\A ,k, n+1 }(x,y)-\LR ^{\A ,k+1,n} (x,y)
\Bigr|^2
\\\notag \leq&
%\Bigl(
\int _{\AAAA _{k}\times\AAAA _{k+1}} 
|\K ^{\A , k}(x,z) \KR ^{\A , n}(w,y)|^{2} 
\mmm(dz)\mmm(dw) 
%\Bigr)^{\frac{1}{2}}
%
%\Bigl(
\int _{\AAAA _k \times \AAAA _{k+1}} 
|\KR (z,w)-\K (z,w)|^2
\mmm(dz)\mmm(dw)
%\Bigr)^{\frac{1}{2}}
.\end{align}
Hence, 
\begin{align}\label{:31k}&
\int_{\AAAA _0 \times \AAAA _m }\Bigl|
\LR  ^{\A ,k,n+1 }(x,y)-\LR ^{\A , k+1 , n }  (x,y) 
\Bigr|^2 \mmm(dx) \mmm(dy)
\\ \notag \leq &
\int _{\AAAA _0 \times \AAAA _k } | \KAk (x,z) |^2 \mmm(dx) \mmm(dz) 
\int _{\AAAA _{k+1} \times \AAAA _{m}} | \KR ^{\A , n}(w,y) |^2 \mmm(dw) \mmm(dy)
\\ \notag &\times 
\int _{\AAAA _{k} \times \AAAA _{k+1}} |\KR(z,w)-\K(z,w)|^2 \mmm(dz)\mmm(dw)
.\end{align}
Recall that $ k+n=m-1$. 
Then $ 0 \leq n \leq m-1 $. 
Let $ \A ^{\prime} =  \AAAA _{k+1}\times  \cdots  \times \AAAA _m $ 
and $ (k ^{\prime},n ^{\prime})$ be such that $ k ^{\prime} + n ^{\prime} =n$. 
Then by replacing $ m$ by $ n$ in \eqref{:31i} we have 
\begin{align}\label{:31l}
\sup_{R} \int _{\AAAA _{k+1}\times \AAAA _m} 
|\LR ^{\A ^{\prime} , k ^{\prime} , n ^{\prime} } (w,y) |^2 \mmm(dw)\mmm(dy) < \infty
.\end{align}
Take $ (k ^{\prime},n ^{\prime}) = (0,n)$. 
Then $ \LR ^{\A ^{\prime} , 0 , n} (x,y) = \KR ^{\A ,n}(x,y)$ by \eqref{:31g}. 
Hence from \eqref{:31l} 
\begin{align}\notag %\label{:}&
\sup_{R} \int _{\AAAA _{k+1}\times \AAAA _m} 
|\KR ^{\A , n} (w,y) |^2 \mmm(dw)\mmm(dy) < \infty
.\end{align}
From this, \eqref{:31e}, and \lref{l:31} \thetag{2}, the last term in \eqref{:31k} goes to zero as $ R \to \infty$. 
Therefore, we see that 
\begin{align}\notag& 
\Bigl(
\int_{\AAAA _0 \times \AAAA _m }
\Bigl|
\LR ^{\A , k , n+1 } (x,y) - \KAm (x,y)
\Bigr|^2 
\mmm  (dx)\mmm (dy)
\Bigr)^{\frac{1}{2}}
\\ \leq & \notag 
\sum_{p=0}^{n}
\Bigl(
\int_{\AAAA _0 \times \AAAA _m }
\Bigl|
\LR ^{\A , k+p , n+1-p}  (x,y) - \LR  ^{\A ,k+p+1, n-p}(x,y)
\Bigr|^2 \
\mmm(dx) \mmm(dy)
\Bigr)^{\frac{1}{2}}
\\ \notag \to & 
0  \text{ as } R \to \infty
.\end{align}
Hence \eqref{:31h} holds for $ m $. 

We deduce \eqref{:31i} for $ m $ from \eqref{:31h} for $ m $ immediately. 

We now apply \eqref{:31i} to obtain \thetag{3}. Let $ \sigma =(1,2,\ldots,m)$. 
\begin{align}\label{:31m}&
\int_{\mathbb{A}  } 
\Bigl\{
\prod_{p=1}^m 
\KR (x_p,x_{\sigma(p)})
-
\prod_{p=1}^m 
\K (x_p,x_{\sigma(p)})
\Bigr\}
\mmm ^{m}(d\mathbf{x})
\\ \notag =&
\sum _{k=0}^{m-1}
\int_{\AAAA _m  } 
\Bigl\{
\LR ^{\A , k ,m-k} (x,x)
-
\LR ^{\A ,k+1, m-k-1} (x,x)
\Bigr\}
\mmm (dx)
.\end{align}
Let $ k+n=m$ and $ n \geq1$. 
Then 
\begin{align}\notag & %\label{:31n}&
\int_{\AAAA _m  } 
\Bigl\{
\LR ^{\A , k ,n } (x,x)
-
\LR ^{\A ,k+1, n-1} (x,x)
\Bigr\}
\mmm (dx)
\\\notag =&
\int _{\AAAA _m}
\Bigl(
\int _{\AAAA _k \times \AAAA _{k+1}}
\Bigl\{
\K ^{\A , k} (x,z) \KR (z,w) \KR ^{\A , n-1}(w,x)
\\\notag &
\quad \quad \quad \quad \quad \quad \quad 
-
\K ^{\A , k} (x,z) \K (z,w) \KR ^{\A , n-1}(w,x)
\Bigr\}
\mmm (dz)\mmm(dw)
\Bigr)
\mmm (dx)
\\\notag  =&
\int _{\AAAA _k \times \AAAA _{k+1}}
\int _{\AAAA _{m}}
\K ^{\A , k}(x,z) \KR ^{\A , n-1}(w,x)
\mmm(dx)
\Bigl(\KR (z,w)-\K (z,w)\Bigr)
\mmm (dz)\mmm(dw)
.\end{align}
%$ \A ^{\prime}=\AAAA_k \times \AAAA _{k-1}\times \cdots \times \AAAA _1 \times \AAAA _m \times \AAAA _{m-1} \times \AAAA _{m-(n-2)} \times \AAAA _{k+1}$.
By the Schwarz inequality, 
\begin{align} \notag &
\Bigr|
\int_{\AAAA _m  } 
\Bigl\{
\LR ^{\A , k ,n } (x,x)
-
\LR ^{\A ,k+1, n-1} (x,x)
\Bigr\}
\mmm (dx)
\Bigl|
\\\notag \leq&
\Bigl(
\int _{\AAAA _k \times \AAAA _{k+1}}
\Bigl|
\int _{\AAAA _{m}}
\K ^{\A , k}(x,z) \KR ^{\A , n-1}(w,x)
\mmm(dx)
\Bigr| ^2
\mmm (dz)\mmm(dw)
\Bigr)^{\frac{1}{2}}
\\\notag &\times
\Bigl(
\int _{\AAAA _k \times \AAAA _{k+1}}
|\KR (z,w)-\K (z,w)|^2
\mmm (dz)\mmm(dw)
\Bigr)^{\frac{1}{2}}
.\end{align}
Recall that $ k+n=m$. Then $ k+n-1=m-1$. 
From \eqref{:31i} for $ m-1 $ and \lref{l:31} \thetag{2}, the last term goes to zero as $ R \to \infty$. 
This combined with \eqref{:31m} implies \eqref{:31x}. 
\end{proof}

Let $ \lambdal $ be the counting measure on $ \IIIl $ as before. 
We can regard $ \KFl $ as an operator on 
$ L^2(\IIIl  ,\lambdal ) $
%$ \map{\KFl }{L^2(\IIIl ,\ml )}{L^2(\IIIl ,\ml )}$ 
such that $ \KFl \p  (\ii ) = \sum_{\jj \in \IIIl } \KFl (\ii ,\jj ) \p  (\jj ) 
$. 
We now prove that the $ (\KFl ,\lambdal ) $-determinantal point 
process $ \nul $ exists. 
\begin{lmm} \label{l:32} 
Let $ \mathrm{Spec}( \KFl  ) $ be the spectrum of $ \KFl $. Then  
\begin{align}\label{:32a}&
 \mathrm{Spec}( \KFl  ) \subset [0,\infty)
.\end{align}
If $ \alpha <0 $, 
\begin{align}\label{:32aa}&
 \mathrm{Spec}( \KFl  ) \subset [0,-\frac{1}{\alpha}]
.\end{align}
In particular, 
there exists a unique, $\alpha$-determinantal point process $ \nuFl $ on $ \IIIl $
associated with $ (\KFl ,\lambdal ) $. 
\end{lmm}
\begin{proof}
Recall that $\6 $ is an orthonormal basis of $ L^2(\Sit ,\mmm )$. 
Let $ \map{U}{L^2(\Sit ,\mmm )}{L^2(\IIIl , \lambdal  )}$ be the unitary operator such that 
$ U (\flii ) = e_{\ell ,i}$, where 
$ \{e_{\ell ,i} \}_{i\in\IIIl }$ 
is the canonical orthonormal basis of $L^2(\IIIl , \lambdal  )$. 
Then by \lref{l:31} we see that $  \KFl  = U \K U^{-1}$. 
Hence $ \KFl $ and $ \K $ have the same spectrum. 
We thus obtain \eqref{:32a} and \eqref{:32aa} from \thetag{A1}.  
Because $ \KFl  $ is Hermitian symmetric, 
the second claim is clear from \eqref{:32a}, \eqref{:32aa}, \thetag{A1}, 
and Theorem 1.2 of 
\cite{s-t.jfa}. 
\end{proof}

\begin{lmm} \label{l:33} 
Let $ \Bli = \mathrm{supp} (\flii )$ be as in \eqref{:21t}.  
Then, for $ \ii , \jj  \in \IIIl $ and $ \AAAA \in \Fl  $, 
\begin{align}\label{:33a}&
\int_{ \AAAA } \flii (x){\fljj (x)} \mmm (dx) =
\begin{cases} 1 & (\ii = \jj ,\ \Bli  \subset \AAAA )
\\
0 & (\text{otherwise})
\end{cases}
.\end{align}
\end{lmm}
\begin{proof} 
We recall that $\Bli $ is the support of $ \flii $ by \eqref{:21t}. 
Suppose $ \ii = \jj $ and $ \Bli \subset \AAAA $. Then 
from \eqref{:30a} we obtain 
\begin{align}\label{:33b}&
\int_{ \AAAA } \flii (x){\fljj (x)} \mmm (dx) = 
\int_{\Sit } \flii (x){\flii (x)} \mmm (dx) = 1 
.\end{align}
Suppose that $ \ii = \jj $ and that $ \Bli \not\subset \AAAA $. 
Then, using $ \AAAA \in \Fl  $, \eqref{:20j}, and \eqref{:21u}, 
we deduce that $ \Bli \cap \AAAA = \emptyset $. 
Because $ \Bli = \mathrm{supp} (\flii ) $, we obtain 
\begin{align}\label{:33c}&
\int_{ \AAAA } \flii (x){\fljj (x)} \mmm (dx) = 0 
.\end{align} 
Finally, suppose $ \ii \not= \jj $. Because $ \AAAA \in \Fl  $,  
we see that $ \Bli \subset \AAAA $ or $ \Bli \cap \AAAA = \emptyset $. 
The same also holds for $ \Blj $. In any case, we obtain \eqref{:33c} 
from \eqref{:30b}.  
%because  $ \flii $ and $ \fljj $ are orthogonal. 
From \eqref{:33b} and \eqref{:33c}, we obtain \eqref{:33a}. 
\end{proof}

\noindent {\em Proof of \tref{l:21}. }
Let $ \mathbb{A}=  
\AAAA _1\ts \cdots \ts \AAAA _m $ 
as in \tref{l:21}. Then, 
because $\AAAA _n \in \Dl  $ for all $ n = 1,\ldots,m $, we deduce 
from \eqref{:21z}, \eqref{:20b}, and \eqref{:10a}  that 
\begin{align}\label{:34a} 
\int_{\mathbb{A} } 
 \rhol ^m(\mathbf{x}) \mmm ^{m} & (d\mathbf{x}) 
= 
\int_{\mathbb{A} } 
 \deta[\K(x_p,x_q)]_{p,q=1}^{m}
\mmm ^{m}(d\mathbf{x}) 
,\end{align}
where $ \mathbf{x}=(x_1,\ldots,x_m)$. 
From a straightforward calculation and \lref{l:31}, we obtain 
\begin{align}\label{:34p}&
\int_{\mathbb{A}  } 
\deta[
\K(x_p,x_{q})
]_{p,q=1}^m
\mmm ^{m}(d\mathbf{x}) \quad 
\\ 
=\notag & 
\int_{\mathbb{A}  } 
\sum_{\sigma \in \mathfrak{S}_m } 
\alpha ^{m-\nu(\sigma)}
\prod_{p=1}^m %\prod_{s=1}^{l_r}
\K (x_p,x_{\sigma(p)})
\mmm ^{m}(d\mathbf{x}) 
\\ 
=\notag & 
\sum_{\sigma \in \mathfrak{S}_m } 
\alpha ^{m-\nu(\sigma)}
\int_{\mathbb{A}  } 
\prod_{p=1}^m %\prod_{s=1}^{l_r}
\K (x_p,x_{\sigma(p)})
\mmm ^{m}(d\mathbf{x}) 
\\ =\notag & 
\sum_{\sigma \in \mathfrak{S}_m } 
\alpha ^{m-\nu(\sigma)}
\limi{R}
\int_{\mathbb{A}  } 
\prod_{p=1}^m %\prod_{s=1}^{l_r}
\KR (x_p,x_{\sigma(p)})
\mmm ^{m}(d\mathbf{x}) 
,\end{align}
where $ \KR $ is defined by \eqref{:31y}. 
We note that $ \cup_{i=1}^m \Ai $ is relatively compact. 
Hence the last line in \eqref{:34p} follows from 
\lref{l:31} \thetag{3}. 
\begin{align}\label{:34t}&
\int_{\mathbb{A}  } 
\prod_{p=1}^m %\prod_{s=1}^{l_r}
\KR (x_p,x_{\sigma(p)})
\mmm ^{m}(d\mathbf{x}) 
\\ =\notag & 
\int_{\mathbb{A}  } 
\prod_{p=1}^m 
\Big(
\sum_{\iip \in \III (\ell ; R )   } 
 \KFl (\iip ,\jjp ) 
\fliip (x_p) 
{\fljjp (x_{\sigma (p)})}
 \Big)
\mmm ^{m}(d\mathbf{x}) 
\\ =\notag &
%\limi{R}
\int_{\mathbb{A}  } 
\Big(
\sum_{\iii ,\,  \jjj \in \III (\ell ; R ) ^m   } 
%% \sum_{\iip , \jjp \in \III (\ell ; R )  } 
\prod_{p=1}^m 
 \KFl (\iip ,\jjp ) 
\fliip (x_p) 
{\fljjp (x_{\sigma (p)})}
 \Big)
\mmm ^{m}(d\mathbf{x}) 
= : J(R)
\end{align}
Here, $ \iii = (\ii _1,\ldots,\ii _m) , 
\jjj = (\jj _1,\ldots,\jj _m) \in \IIIl ^m $. 
From  \lref{l:33}, 
\begin{align}\label{:34t2}
J(R) &= 
\int_{\mathbb{A}  } 
\Big(
\sum_{\iii ,\,  \jjj \in \III (\ell ; R ) ^m   } 
%\sum_{\iip , \jjp \in \III (\ell ; R )  } 
\prod_{p=1}^m 
 \KFl (\iip ,\jjp ) 
\fliip (x_p) 
{f_{\ell ,\jj _{\sigma^{-1} (p )}} (x_p) }
 \Big)
\mmm ^{m}(d\mathbf{x}) 
\\ \notag 
= &
%\limi{R}
\int_{\mathbb{A}  } 
\Big(
\sum_{\iii \in \III (\ell ; R ) ^m \cap \IIIl (\mathbb{A})  } 
%\sum_{\iip \in \III (\ell ; R )  } 
\prod_{p=1}^m %\prod_{s=1}^{l_r}
 \KFl (\iip ,\ii _{\sigma (p)} ) 
|\fliip (x_p) |^2
 \Big)
\mmm ^{m}(d\mathbf{x}) 
\\ \notag 
= &
\sum_{\iii \in \III (\ell ; R ) ^m \cap \III _{\ell} (\mathbb{A})  } 
\prod_{p=1}^m %\prod_{s=1}^{l_r}
 \KFl (\iip ,\ii _{\sigma (p)} ) 
%\\ \notag \to &
\to
\sum_{\iii \in \III _{\ell} (\mathbb{A}) } 
\prod_{p=1}^m %\prod_{s=1}^{l_r}
 \KFl (\iip ,\ii _{\sigma (p)} ) 
\quad \text{ as } R \to \infty 
.\end{align}
The convergence in the last line follows from %%
\lref{l:31} \thetag{2} and the Schwarz inequality. 
Multiplying $ \alpha ^{m-\nu(\sigma)}$ and summing over 
$ \sigma \in \mathfrak{S}_m $ in the last term, we see that 
\begin{align}\label{:34r} & 
%\sum_{\sigma \in \mathfrak{S}_m } 
%\alpha ^{m-\nu(\sigma)}
%
%\limi{R}
%\sum_{\iii \in \III (\ell ; R ) ^m \cap \III _{\ell} (\mathbb{A})  } 
%
%\prod_{p=1}^m 
%\KFl (\iip ,\ii _{\sigma (p)} ) 
%\\ = \notag &
\sum_{\sigma \in \mathfrak{S}_m } \alpha ^{m-\nu(\sigma)}
\sum_{\iii \in \III _{\ell} (\mathbb{A}) } 
\prod_{p=1}^m %\prod_{s=1}^{l_r}
 \KFl (\iip ,\ii _{\sigma (p)} ) 
\\ = \notag &
\sum_{\iii \in \III _{\ell} (\mathbb{A}) } 
\sum_{\sigma \in \mathfrak{S}_m } \alpha ^{m-\nu(\sigma)} 
\prod_{p=1}^m %\prod_{s=1}^{l_r} 
\KFl (\iip ,\ii _{\sigma (p)} ) 
\\ = \notag & 
\sum_{\iii  \in \III_{\ell }(\mathbb{A}) }
\deta[ \KFl (\iip ,\iiq ) 
]_{p,q=1}^m 
\\ = \notag &
\sum_{\iii   \in \III_{\ell }(\mathbb{A}) }
\rhoFl ^m (\iii ) 
.\end{align}
Combining \eqref{:34a}--\eqref{:34r} 
we deduce \eqref{:21b}, which completes the proof. 
\qed

\section{Proof of \tref{l:22} and \tref{l:23} }\label{s:4}

\subsection{Proof of \tref{l:22}} 
Let $ \varrho ^m $ be the $ m $-point correlation function of $\8 $. 
Then it suffices for \eqref{:22a} to prove 
\begin{align}\label{:41z}& 
\rhol ^m (\mathbf{x}) = \varrho ^m (\mathbf{x})  
.\end{align}

From \eqref{:20e} and $ \mathfrak{F}\kl = \sigma [ \Ali ; i\in \Il  ]$, 
we see that $ \rhol ^m $ and $ \varrho ^m $ are $ \Fl ^m $-measurable. 
Let $  m = m_1+\cdots + m_k $. Let 
$\mathbb{A} = \AAAA _1^{m_1}\ts \cdots \ts \AAAA _k^{m_k} \in\Dl ^m $ 
such that 
$ \AAAA _p \cap \AAAA _q = \emptyset $ if $ p\not=q $. 
Let 
$ \iii = (i_n)_{n=1}^m = 
(\iii _1,\ldots,\iii _k)\in \IIIl ^m  $ 
such that $ \mathbf{i}_n \in \III (\ell )^{m_n} $. 
From \tref{l:21}, we see that 
\begin{align}\label{:41a}
\int_{\mathbb{A}} \rhol ^m(\mathbf{x}) \mmm ^{m}(d\mathbf{x}) = 
\sum_{\iii \in \III_{\ell }(\mathbb{A}) } \rhoFl ^m (\iii ) 
.\end{align}
By the definition of correlation functions, 
\eqref{:22w}, and \eqref{:22x}, we see that 
\begin{align}\label{:41d}
\sum_{\iii \in \III_{\ell }(\mathbb{A}) } 
\0  
=&
%E ^{\nul }[
\int _{\IIl }\prod_{n=1}^k
\frac{\mathsf{i}  (\III_{\ell }(\AAAA _n)) ! }
{(\mathsf{i}  (\III_{\ell }(\AAAA _n)) -m_n)!} 
\nul (d\mathsf{i})
\\ \notag 
=& %E^{\8 }[
\int_{\Ssf }
\prod_{n=1}^k
\frac{\sss  (\AAAA _n ) ! }{(\sss  (\AAAA _n ) -m_n)!}  
\8 (d\sss )
%]
\\ \notag 
=&
\int_{\mathbb{A}} 
\varrho ^m (\mathbf{x}) \mmm ^m (d\mathbf{x})
.\end{align}
Combining \eqref{:41a} and \eqref{:41d}, we deduce that 
\begin{align}\label{:41e}&
\int_{\mathbb{A}} 
\rhol ^m (\mathbf{x}) \mmm ^m (d\mathbf{x}) =
\sum_{\iii \in \III_{\ell }(\mathbb{A}) } \0  
= 
\int_{\mathbb{A}} 
\varrho ^m (\mathbf{x}) \mmm ^m (d\mathbf{x})
.\end{align}
From \eqref{:41e}, we obtain \eqref{:41z}. This completes the proof 
of \tref{l:22}. 
\qed

\subsection {Proof of \tref{l:23}} 
Let $ \mathsf{A} \in \Gl$. 
From \tref{l:22} and 
regular conditional probability of $ \7 $ with respect to $ \sigma [ \iotal ]$, we see that 
\begin{align}\label{:l23}
\mul (\mathsf{A}) 
=& 
(\7 ) \circ \ulabl ^{-1} |_{ \Gl } (\mathsf{A}) 
\\ 
=& \notag
\int_{\IIl }(\7 ) \circ \iotal ^{-1} (d\mathsf{i}) 
\, \, 
\7 ( \ulabl ^{-1}    (\mathsf{A})  | \iotal (\omega ) = \mathsf{i}) 
\\ 
=& \notag
\int_{\IIl }(\7 ) \circ \iotal ^{-1} (d\mathsf{i}) 
\, \, 
\7 ( \kappal ^{-1} \circ \ulabi ^{-1}    (\mathsf{A})  | \iotal (\omega ) = \mathsf{i}) 
\\ 
=& \notag
\int_{\IIl }\nuFl (d\mathsf{i}) \,  \mmFii \circ \ulabi ^{-1}    (\mathsf{A}) 
\\ 
=& \notag
\int_{\IIl }\nuFl (d\mathsf{i}) \,  \mmFii ^{\ulab }    (\mathsf{A}) 
.\end{align}
Here the forth line in \eqref{:l23} follows from the fact 
$ \ulabl ( \omega ) = \ulabi ( \kappal ( \omega ) )$ 
for each $ \omega  \in \4 $ with $ \iota _{\ell} (\omega)= \mathsf{i}$. 
From \eqref{:l23}, we obtain \tref{l:23}. 
\qed

\section*{Acknowledgment}
We express sincere thanks to Professor Ryoki Fukushima for his valuable comments 
in 15TH Stochastic Analysis on Large Scale Interacting Systems 
at Tokyo University. 
That was whether the results for the determinantal point processes in \cite{o-o.tail} also holds for the $ \alpha$-determinantal point processes.

%%%%%%%%%%%%%%%%%%%%%%%%%%%%%%%%%%%%%%%%%%%%%%%%%%

\end{document}